\documentclass{cimart}

\title{
    On the Nowicki conjecture for the two-generated free Lie algebra
    }

\author{
    Lucio Centrone, {\c S}ehmus F{\i}nd{\i}k and Manuela da Silva Souza
    }

\authorinfo[
    L. Centrone]{University of Bari Aldo Moro, Bari, Italy
    }{%
    lucio.centrone@uniba.it
    }

\authorinfo[
    {\c S}. F{\i}nd{\i}k]{
    \c{C}ukurova University, Adana, Turkey}{%
    sfindik@cu.edu.tr
    }

\authorinfo[
    M. da S. Souza]{
    Federal University of Bahia, Salvador, Brazil}{%
    manuela.souza@ufba.br
    }

\abstract{%
    Let $K[X_n]=K[x_1,\ldots,x_n]$ be the polynomial algebra in $n$ variables over a field $K$ of characteristic zero.
A locally nilpotent linear derivation $\delta$ of $K[X_n]$ is called Weitzenb\"ock due to his well known result from 1932
stating that the algebra of constants of $\delta$ defined by $\text{\rm ker}(\delta)=K[X_n]^{\delta}$ is finitely generated.
The explicit form of a generating set of $K[X_n,Y_n]^{\delta}$ was conjectured by Nowicki in 1994 in the case $\delta$ was such that $\delta(y_{i})=x_{i}$, $\delta(x_i)=0$, $i=1,\ldots,n$.
Nowicki's conjecture turned out to be true and, recently, has been applied to several relatively free associative algebras.
In this paper, we consider the free Lie algebra $\mathcal{L}(x,y)$ of rank $2$ generated by $x$ and $y$ over $K$ and we assume the Weitzenb\"ock derivation $\delta$ sending $y$ to $x$, and $x$ to zero.
We introduce the idea of pseudodeterminants and we present a characterization of Hall monomials that are constants showing they are not so far from being pseudodeterminants. We also give a complete list of generators of the constants of degree less than 7 which are, of course, pseudodeterminants.
    }



\VOLUME{33}
\YEAR{2025}
\ISSUE{3}
\NUMBER{9}
\DOI{https://doi.org/10.46298/cm.15036}

\begin{document}

\section{Introduction}

Let $K[X_n]$ be the polynomial algebra with generating set $X_n=\{x_1,\ldots,x_n\}$, $n\geq2$ over a field $K$ of characteristic zero. It is the free algebra of rank
$n$ in the variety of unitary commutative algebras over $K$. Let $H$ be a subgroup of the general linear group
$GL_n(K)$. Then a polynomial $p\in K[X_n]$ is $H$-\textit{invariant} if it is preserved under the action of each element of $H$. The vector space
$K[X_n]^H$ of all $H$-invariants is called the algebra of $H$-invariants.

The question whether the algebra $K[X_n]^H$ of invariants
is finitely generated for every subgroup $H$ of $GL_n(K)$ is a special case of the Hilbert's fourteenth problem
suggested by the German mathematician David Hilbert in 1900 at the International Congress of Mathematicians in Paris.
Although the answer to Hilbert's question turned out to be negative in general (see for instance the paper by Nagata \cite{Ngt} (1958)), some remarkable affirmative cases have been handled as well.
Among them is the approach of Weitzenb\"ock \cite{W}, where he considered the locally nilpotent linear derivations $\delta$ of the algebra $K[X_n]$.
Then the kernel $\text{\rm ker}(\delta)=K[X_n]^{\delta}$ of the derivation $\delta$ is an algebra so called the algebra of constants.
He showed that the algebra $K[X_n]^{\delta}$ is finitely generated as an algebra, which is equal to the algebra 
$K[X_n]^{UT_2(K)}$ of invariants of 
the unitriangular group given by ${UT_2(K)}=\{\text{exp}(c\delta)\mid c\in K\}$.
Thus the study of the algebra of constants inherits methods from the classical invariant theory. 
The books by Nowicki \cite{N}, Derksen and Kemper \cite{DK},
and Sturmfels \cite{St} are suggested for the readers interested in numerical aspects of algebras of constants and general invariant theory.

The next question after the approach of Weitzenb\"ock has become the explicit forms of those generators. Nowicki conjectured in his book \cite{N} that the algebra $K[X_n,Y_n]^{\delta}$
is generated by $x_1,\ldots,x_n$, and $x_iy_j-y_ix_j$, $1\leq i<j\leq n$,
assuming that $\delta(y_{i})=x_{i}$, $\delta(x_{i})=0$, $i=1,\ldots,n$.
The Nowicki's conjecture was proved by Khoury \cite{K1,K2}, 
Drensky and Makar-Limanov \cite{DML}, Kuroda \cite{Kuroda},
Bedratyuk \cite{Bed} using different techniques. See also the paper \cite{D2} by Drensky for a generalized version of Nowicki's conjecture in which the author proves, as a consequence of the paper by Kuroda of 2009, that the algebra of constants of $K[X_n,Y_n]^\delta$ is generated by $X_n$ and the “determinants" $u_{i,j}:=f_iy_j-f_jy_i$'s, when the derivation $\delta$ is such that $\delta(x_i)=0$ and $\delta(y_i)=f_i$, $i=1,\ldots,n$, where the $f_i$'s are nonconstant polynomials of $K[X_n]$.

A natural generalization of the problem is to study the generators of the algebra of constants of (relatively) free associative or Lie algebras.
In \cite{DG} Drensky and Gupta studied Weitzenb\"ock derivations acting on the relatively free algebra in a variety $\mathfrak{V}$. They proved that
if  $\mathfrak{V}$ contains $UT_2(K)$, then the algebra of constants is not finitely generated. 
It is also known by a result of Drensky in \cite{D1} that 
if $UT_2(K)$ does not belong to $\mathfrak{V}$, then  the algebra of constants is finitely generated.
One may list the recent works as follows. Let $\delta$ be a Weitzenb\"ock derivation of the free metabelian Lie algebra $F_n$
and of the free metabelian associative algebra $A_n$ of rank $m$.
Dangovski et al. \cite{DDF,DDF1} showed that 
the algebras of constants $F_{n}^{\delta}$ and $A_{n}^{\delta}$ are not finitely generated as an algebra
except for some trivial cases. 
Also Drensky and one of the authors \cite{DF1}
considered the Nowicki conjecture for free metabelian Lie algebras, and they gave explicit forms of the generators.
As a continuation of this approach, two among the authors \cite{CF}
solved the problem for several relatively free algebras. See also the paper \cite{cedufi} for a survey toward this argument.

The goal of this paper goes in this direction but in a free Lie algebra setting. We recall, by the well known dichotomy, a variety of Lie algebras either satisfies the
Engel condition (hence it is nilpotent by \cite{zel1}) or contains the metabelian variety $\mathfrak{A}^2$ consisting of all solvable Lie algebras of class 2 which is defined by the identity $[[x_1, x_2], [x_3, x_4]] = 0$. Since the finitely generated nilpotent Lie algebras are finite dimensional, the problem for the finite generation of the algebras of constants of relatively free nilpotent Lie algebras is solved trivially. In this paper we present the algebra of constants of the free Lie algebra $\mathcal{L}(x,y)$ in two generators. We make use of the definition of pseudodeterminants. The name pseudodeterminants is due to some own algebraic properties and their similarities with the determinant-like generators used by Makar-Limanov and Drensky in their proof of Nowicki's conjecture for the free commutative algebra. We give a characterization of the monomials which are constants. It turns out they can be monomials in the algebra generated by the pseudodeterminants or they are “very close" to being pseudodeterminants. We finally give some examples of an explicit set of generators for constants of small degree (up to degree 7) and based on these experimental results, we conjecture a stronger proposition:

\[\textit{The algebra of constants of the free Lie algebra of rank 2 is generated by}\]\[\textit{$x$ and the pseudodeterminants that are constants.} \]

\section{Preliminaries}

Let us fix some notations and some basic facts in our Lie setting.  

\begin{definition}A Lie algebra $L$ is a vector space over the field $K$ with a $K$-bilinear map $[\cdot,\cdot]:L \times L \rightarrow L$ such that \[\text{$[a,a]=0$ and $[[a,b],c]+[[c,a],b]+[[b,c],a]=0$ }\] for every $a,b,c\in L$.
\end{definition}

The second of the relations above is called \textit{Jacobi's identity}. For the sake of simplicity of notations we shall omit brackets in case of left-normed products, that is, we define 
$[a_1, a_2, \ldots, a_{n-1}, a_n] = [[[\ldots [a_1, a_2], \ldots], a_{n-1}], a_n]$
for all $a_1$, $\ldots$, $a_n \in  L$. An immediate consequence of Jacobi's identity is the relation
\begin{equation}\label{identidade_importante}
[a, b, c] = [a, c, b] - [a, [c, b]]
\end{equation}
for all $a$, $b$, $c \in L$. In particular $[a, b, a, b] = [a, b, b, a]$. We shall denote by $L'$ the Lie subalgebra $[L,L]$ of $L$, which is called the \textit{commutator ideal} of $L$.

As well as for the associative case, in the Lie setting we can construct a free object, called the free Lie algebra. If $X$ is a set of indeterminates we are allowed to consider $\mathcal{L}(X)$ the free Lie algebra over the field $K$ generated by the set $X$. In the sequel we shall use linear basis of the free Lie algebra in order to perform our computations. Indeed linear basis of the free Lie algebra have been studied by several authors using different techniques. We shall use one of them in particular which construction will be presented below.

Suppose $X$ is linearly ordered, then using an induction on the degree of Lie monomials, we construct \textit{basic words} starting from the ``smallest'' basic words. Each $x_i\in X$ is a basic word of degree 1. Then let $n>1$ and suppose that for every $d<n$ all basic words of degree $d$ are defined and ordered. Suppose also that provided $w,u$ and $v$ are basic words of degree less than $n$, then $w=[u,v]$ implies $w>v$. Additionally the monomial $q$ of degree $n$ is basic if, whenever $q=[q_1,q_2]$, then $q_1,q_2$ are basic words and $q_1>q_2$ and the decomposition $q=[[q',q''],q_2]$ implies $q''\leq q_2$. Any monomial not respecting the previous axioms is to be considered not basic. We have the next result. 

\begin{theorem} [\hspace{-.01cm}\cite{hal1}]
The basic words in $X$ form a basis of the free Lie algebra $\mathcal{L}(X)$.
\end{theorem}

Instead of basic words we shall use the expression \textit{Hall basis}. We will denote by $\mathcal{B}$ the Hall basis of the free Lie algebra $\mathcal{L}(x, y)$. The order used here is the deg-lexicographic one: $f \leq g$ if, and only if, $\mbox{deg}(f) < \mbox{deg}(g)$ or $\mbox{deg}(f) = \mbox{deg}(g)$ and $f \leq_{\mbox{lex}} g$ where $\leq_{\mbox{lex}}$ is the lexicographic order and $x < y$, where we compare lexicographically the words $f$ and $g$ after deleting the brackets.

The first basis of a free Lie algebra was found by Hall in \cite{hal1} whereas Shirshov in \cite{shi2} and Lyndon in \cite{lyn1} constructed the basis of a free Lie algebra consisting of the so called \textit{Lyndon-Shirshov words} (see for example Bahturin's book \cite{Bah}). In a former work \cite{shi3} Shirshov was able to develop the composition method for Lie
algebras which was refined and rewritten in a modern language setting by Bokut in \cite{bok1}. The smart contribution of Shirshov is highlighted in \cite{shi2} where he suggested a path for choosing bases for free Lie algebras generalizing those of Hall's basis and of Lyndon-Shirshov's basis. Other  examples  of  bases  for  free  Lie  algebras  were  found  by  Bokut  in \cite{bok2}, by Reutenauer in \cite{reu1}, by Blessenohl and Laue in \cite{bll1}, by Bryant, Kovacs and St\"{o}hr in \cite{bks1}, by Guilfoyle and St\"{o}hr in \cite{gus1} and by Chibrikov in \cite{chi1} in which the author shows a right normed monomial basis for the free Lie algebra. For more details on the theory of free Lie algebras we address the reader to the book of Reutenauer \cite{reuten} or the paper \cite{reu1} by the same author.

We will also use a well celebrated result by Shirshov \cite{shi4} and Witt \cite{wit1} which separately obtained the same result.

\begin{theorem}\label{shirshovwitt}
Every Lie subalgebra of a free Lie algebra is free. Moreover, every finite dimensional Lie subalgebra of a free Lie algebra has dimension 1.
\end{theorem}

Now we will give a look toward the notion of derivation for a given algebra. 

\begin{definition}
If $A$ is any algebra (associative or not) over a field $K$, a derivation of $A$ is a $K$-linear map $\delta:A\rightarrow A$ so that $\delta(ab)=\delta(a)b+a\delta(b)$ for every $a,b\in A$.
\end{definition}

Of course, every derivation of a free algebra is completely determined by the value it takes on its generators. Moreover we shall denote by 
\[A^\delta:=\{a\in A|\delta(a)=0\}\]
the algebra of constants of $A$.

As mentioned before, throughout the paper we deal with one particular locally nilpotent derivation of $\mathcal{L}(x,y)$. We recall a locally nilpotent derivation of an algebra $A$ is a derivation $\delta$ such that $\delta^n(a)=0$ for every $a\in A$ with $n$ depending on $a$. Here we shall study the behaviour of $\mathcal{L}(x,y)^\delta$, where $\delta$ is the derivation of $\mathcal{L}(x,y)$ sending $y\mapsto x$ and $x\mapsto0$. Notice that $\delta$ is a locally nilpotent derivation of $\mathcal{L}(x,y)$.

\section{Pseudodeterminants}

As already mentioned in the introduction, the algebra of constants of $\mathcal{L}(x,y)$ is not finitely generated. In what follows we shall construct an ad hoc algebra for monomials (Hall monomials) that are constants starting from the definition of pseudodeterminants. Every field here is assumed to be of characteristic 0.

We recall the first formulation of Nowicki's conjecture stated the algebra of constants $K[X_n,Y_n]^\delta$, where $X_n$ and $Y_n$ are two finite sets of $n$ commutative variables, is generated by the variables $X_n$ and the
\[
u_{ij}=\left |\begin{array}{cc}
x_i & x_j\\
y_i & y_j	
\end{array}\right|.
\]

In \cite{DML} Drensky and Makar-Limanov found a uniformly looking explicit set of defining relations of the algebra of constants $K[X_n,Y_n]^\delta$ which corresponds to the reduced Gr\"obner basis
of the related ideal of $K[X_n,u_{ij} | 1\leq i< j\leq n]$. The basic idea of the paper is the construction of objects which are similar to the \text{\it determinants} $u_{ij}$.

From now on we shall denote $L:=\mathcal{L}(x,y)$ and we shall construct some peculiar polynomials that will be useful in figuring out the constants of $L$.
We consider the following elements of $L$, \[U^{(m,k)}_{A,B}:=
\left|
\begin{array}{cc}
\delta^m(A) & \delta^{k}(A)\\
\delta^m(B) & \delta^{k}(B)	
\end{array}
\right|
=[\delta^m(A),\delta^{k}(B)]-[\delta^{k}(A),\delta^m(B)],\] where $A$ and $B$ are Hall monomials.

\begin{definition}
We shall call $U^{(m,k)}_{A,B}$ \text{\it pseudodeterminant} of degree $(m,k)$ in the Hall monomials $A$ and $B$.
\end{definition} 

First of all, notice the following facts.

\begin{remark}\label{good}
Notice that
\[
\delta(U^{(m,k-1)}_{A,B})=U^{(m,k)}_{A,B}+U^{(m+1,k-1)}_{A,B}
\]
and in particular
$\delta(U^{(m,k-1)}_{A,B})=U^{(m,k)}_{A,B}$ if $\delta^{m+1}(A)=\delta^{m+1}(B)=0$. Moreover, we have $U^{(m,k)}_{A,B}=-U^{(k,m)}_{B,A}$, so we are allowed to consider $m\geq k$.

\end{remark}

\begin{remark}\label{import}
Of course we have $U^{(k,0)}_{A, A} = 2[\delta^k(A), A]$. In particular  if $A$ and $B$ are Hall monomials such that $\delta^k(A)=B$,
then $U^{(k,0)}_{A,A}=2[B,A]$. 
\end{remark}

The pseudodeterminants of degree $(k,0)$ describe completely a special class of polynomials as we can see below.

\begin{proposition}\label{lemma3.4}
Let $f=[p,\delta^k(p)]$, where $p\in L$, then $f$ is a linear combination of pseudodeterminants of degree $(k,0)$.
\end{proposition}
\proof
Let $p=\sum\alpha_iM_i$, where the $M_i$'s are Hall monomials and $\alpha_i \in K$. Then \begin{equation}\label{decomp}
    f=\left[\sum\alpha_iM_i,\sum\alpha_j\delta^k(M_j)\right]=\sum_i\sum_j\alpha_i\alpha_j[M_i,\delta^k(M_j)].
\end{equation} Notice that in Equation (\ref{decomp}), for any fixed couple $i,j$, we have the summand \[\alpha_i\alpha_j([M_i,\delta^k(M_j)]+[M_j,\delta^k(M_i)])=\alpha_i\alpha_j([M_i,\delta^k(M_j)]-[\delta^k(M_i),M_j])=-\alpha_i\alpha_jU^{(k,0)}_{M_j,M_i},\] and if $i=j$, then \[[M_i,\delta^k(M_i)]=-\frac{1}{2}U_{M_i,M_i}^{(k,0)}.\] Hence \[f=-\sum_{i<j}\alpha_i\alpha_jU^{(k,0)}_{M_i,M_j}-\frac{1}{2}\sum_i\alpha_i^2U_{M_i,M_i}^{(k,0)}\] and we are done.
\endproof

\section{On constant monomials}


We investigate on the structure of Hall monomials which are constants. We believe this will share light on the structure of constants in general.

The next easy and well known result will be of help in our work. We would like to give a simple proof of it in order to highlight the strict relations between the Lie algebra structure of $L$ and its elements. In fact it is a direct consequence of the Shirshov-Witt Theorem.

\begin{proposition}\label{primality}
Let $f,g\in L$ such that $[f,g]=0$, then $f=\alpha g$ for some $\alpha \in K$.
\end{proposition}
\begin{proof}
Let us suppose $g\neq0$ and do consider $S$ the Lie algebra generated by $f$ and $g$. By Theorem \ref{shirshovwitt} $S$ is free and all its products are 0. Hence the dimension of $S$ as a vector space is less than or equal to 2. This means $S$ is a finite dimensional free Lie algebra. By the second part of Theorem \ref{shirshovwitt} and the fact that $g\neq0$ we get the dimension of $S$ is 1 and we are done.
\end{proof}

The next is a consequence of Proposition \ref{primality} and is of independent interest.

\begin{corollary} If $[g, x] + [\delta(g), y] = 0$ and $g\notin span_K\{x\}$, then $g = \alpha y+\beta x$ for some $\alpha,\beta \in K$ and $\alpha\neq 0$.
\end{corollary}

\begin{proof}  It is sufficient to prove the statement for homogeneous polynomials; suppose $g$ is a homogeneous polynomial. Because $\delta$ is locally nilpotent and $g\notin  span_K\{x\}$, there exists an integer $n \geq 2$ such that $\delta^n(g) = 0$ and $\delta^{n-1}(g) \neq  0$. Since $[g, x] + [\delta(g), y] = 0$, by induction we obtain
\[
n[\delta^{n-1}(g), x] + [\delta^n(g), y] = 0.
\]
Therefore, $n[\delta^{n-1}(g), x] = 0$ and consequently $\delta^{n-1}(g) = \alpha x$ with $\alpha \neq 0$ by Proposition \ref{primality}.
Notice that the non-zero elements $[g, x]$ and $[\delta(g), y]$ are of the same degree as well as $g$ and $\delta(g)$.

If $n=2$, since $g$ is homogeneous, we have $g = \alpha y+\beta x$, for $\alpha,\beta\in K$ and $\alpha\neq0$.
If $n\geq3$, then there exists an element $w\in L$ such that $\delta(w)=\alpha y$ that is impossible.
\end{proof}

Let us denote by $U^{(k)}$ the Lie subalgebra of $L^\delta$ generated by the polynomials that are constants of type $M=[f,g]$ and such that there exists $r\geq k\geq 1$ in order that $\delta^{r-1}(f)=\alpha\delta^{k-1}(g)\neq 0$, by the constants of type $U^{(k,0)}_{A,B}$, where $A$ and $B$ are Hall monomials and by $x$. Then we get the next property of Hall monomials that are constants.

\begin{theorem}\label{monomial}
Let $M = [A, B] \neq 0$ be a monomial of $L^\delta$, where $A$, $B$ are monomials in the Hall basis, then {$M\in U^{(k)}$, where $k=\min\{r,s\}$, $\delta^r(A)=0$, $\delta^s(B)=0$ whereas $\delta^{r-1}(A)\neq0\neq\delta^{s-1}(B)$}. Moreover, $M\in U^{(1)}$ if at least one of $A$ and $B$ belongs to $L^\delta$.
\end{theorem}

\begin{proof}
Let $\delta(M)=\delta([A,B])=0$. Notice that, by induction, for any $n\geq1$ we have the Leibniz rule
\begin{equation}\label{deriv1}
\delta^n(M)=\sum_{0\leq i\leq n}\binom{n}{i}[\delta^{n-i}(A),\delta^i(B)]=0.
\end{equation}
Suppose neither $A$ nor $B$ is a constant and let us fix $r,s\geq2$
such that
\[
\delta^r(A)=0=\delta^s(B), \ \ \ \delta^{r-1}(A)\neq0\neq\delta^{s-1}(B),
\]
where $s\leq r$. Now assume that $n=r+s-2$. Then we get from Equation (\ref{deriv1}):
\begin{align}
0&=\delta^{r+s-2}(M)=\sum_{0\leq i\leq r+s-2}\binom{r+s-2}{i}[\delta^{{r+s-2}-i}(A),\delta^i(B)]\nonumber\\
&=\binom{r+s-2}{s-1}[\delta^{r-1}(A),\delta^{s-1}(B)]+\sum_{i\neq s-1}\binom{r+s-2}{i}[\delta^{{r+s-2}-i}(A),\delta^i(B)].\nonumber
\end{align}
If $i>s-1$, then $\delta^i(B)=0$; otherwise, if $i<s-1$, then ${r+s-2}-i>r-1$, and $\delta^{{r+s-2}-i}(A)=0$. Hence the second sum is zero and we get
$[\delta^{r-1}(A),\delta^{s-1}(B)]=0$.
Because $\delta^{r-1}(A)\neq0$ and $\delta^{s-1}(B)\neq0$, by Proposition \ref{primality}, we have

\begin{equation}\label{goodrel}
 \delta^{r-1}(A)=\alpha\delta^{s-1}(B),   
\end{equation}
for some non-zero $\alpha\in K$, 
which proves the first assertion.

Suppose now $A,B$ are constants. If $M$ has degree 2, then $M$ is a scalar multiple of $[x,y]=\frac{1}{2}U_{y,y}^{(1,0)}$ and we are done. If $\mbox{deg}(M) = 3$, then for some $0\neq\alpha\in K$,
\[
M = \alpha [x, y, x] = \alpha [[x, y], x]=(\alpha/2)[U^{(1,0)}_{y,y}, x]
\]
and consequently $M\in U^{(1)}$. If $\mbox{deg}(M)>3$, since $A$, $B$ are monomials such that $\mbox{deg}(A)$, $\mbox{deg}(B) < \mbox{deg}(M)$, the result follows by induction. Now suppose that $A$ is a constant and $B$ is not a constant. Then, by Proposition \ref{primality},
\[
0 = \delta(M) = [A, \delta(B)],
\]
which implies that $\delta(B) = \beta A$ for some $0\neq\beta\in K$, since $A \neq 0\neq\delta(B)$. Therefore $M = \beta^{-1} [\delta(B), B]$ and the proof follows because $M=\frac{\beta^{-1}}{2}U_{B,B}^{(1,0)}$. 
\end{proof}

\begin{remark}By Proposition \ref{lemma3.4}, in the hypotheses of Theorem \ref{monomial}, we get if $A=\delta^k(B)$ for some $k\geq1$, then $M$ is in the algebra generated by the $ U_{C,C}^{(k, 0)}$ for some $k$. This forces, in order to figure out explicitly how the monomials that are constants look like, to focus on the attention on such a set of constants. If $M=[A,B]$ is a constant but neither $A$ nor $B$ is, with the same hypotheses of Theorem \ref{monomial}, then \[0 = \delta(M)= \delta^{r+ s - 2}(M) = [\delta^{r-1}(A),\delta^{s-1}(B)]=U_{A,B}^{(r-1,s-1)}:\] if some kind of \textit{integration theory} could be developed herein, $M$ looks like very close to being a pseudodeterminant. Unfortunately, the algebras $U^{(k)}$ are very difficult to be described in terms of generators; this justifies the computations of the next section. Our goal now is giving experimental results in order to restrict future studies on a smaller class of generators.
\end{remark}





\section{Lie polynomials of small degree}

In this section, we shall calculate the generators of the constants of small degree. The results show that if $f = f(x, y) \in L^{\delta}$  is a polynomial of degree less than or equal to $7$, then $f $ belongs to the subalgebra generated by $x$, $[y, x]$, $[y, x, y, [y, x, x]]$. In particular, we observe that these polynomials are equal to a multiple scalar of the determinant 
\[
\left|
\begin{array}{cc}
A & \delta(A)\\
B & \delta(B)	
\end{array}
\right|
=[A,\delta(B)]-[\delta(A),B],
\]
when $A = B = y$ and $A = B = [y, x, y]$ respectively.

We shall present two proofs: a direct one that shows, degree by degree of the constants, the algebra the constants belong to, and their vector space structure; a more theoretical one using invariant theory and having a sharply less number of calculations.

\subsection{Direct proof}
We start off with the next.

\begin{lemma}\label{previouslemma}
For all $a \geq 0$ and $b \geq1$,
$$[y, x, \underbrace{x, \ldots, x}_{a}, y, \underbrace{x, \ldots, x}_{b}] = [y, x,  \underbrace{x, \ldots, x}_{a+ b}, y] + \sum_{i = 0}^{b-1}  \binom{b}{i} [y, x, \underbrace{x, \ldots, x}_{a + i}, [y, \underbrace{x, \ldots, x}_{b - i}]]$$
where $\displaystyle\binom{b}{i} = \dfrac{b!}{i! (b-i)!}$ is a binomial coefficient.
\end{lemma}

\begin{proof} The proof will be performed  by induction on $b$. If $b = 1$ the result follows from (\ref{identidade_importante}).
We have that
\begin{eqnarray*}
[y, x, \underbrace{x, \ldots, x}_{a}, y, \underbrace{x, \ldots, x}_{b+1}]& = & [y, x, \underbrace{x, \ldots, x}_{a+ 1}, y, \underbrace{x, \ldots, x}_{b}] + [y, x, \underbrace{x, \ldots, x}_{a}, [y, x], \underbrace{x, \ldots, x}_{b}]\\
& = & [y, x, \underbrace{x, \ldots, x}_{a+ b + 1}, y ] + [y, x, \underbrace{x, \ldots, x}_{a+ b}, [y, x] ]\\
& & + \sum_{i = 0}^{b-1}  \binom{b}{i} [y, x, \underbrace{x, \ldots, x}_{a + 1 +  i}, [y, \underbrace{x, \ldots, x}_{b - i}]] \\
& & + \sum_{j = 0}^{b-1}  \binom{b}{j} [y, x, \underbrace{x, \ldots, x}_{a + j}, [y, \underbrace{x, \ldots, x}_{(b + 1) - j}]]
\end{eqnarray*}
\begin{eqnarray*}
& = & [y, x, \underbrace{x, \ldots, x}_{a+ b + 1}, y ] + [y, x, \underbrace{x, \ldots, x}_{a+ b}, [y, x] ]\\
& & + \sum_{j = 1}^{b}  \binom{b}{j-1} [y, x, \underbrace{x, \ldots, x}_{a + j}, [y, \underbrace{x, \ldots, x}_{(b+1) - j}]]  \\
& & + \sum_{j = 0}^{b-1}  \binom{b}{j} [y, x, \underbrace{x, \ldots, x}_{a + j}, [y, \underbrace{x, \ldots, x}_{(b + 1) - j}]] \\
& = &[y, x,  \underbrace{x, \ldots, x}_{a+ b+ 1}, y] + \sum_{j = 0}^{b}  \binom{b+ 1}{j} [y, x, \underbrace{x, \ldots, x}_{a + j}, [y, \underbrace{x, \ldots, x}_{(b+1) - j}]]
\end{eqnarray*}
where the first equality holds by the induction hypothesis and the last follows by the relation $ \binom{b+ 1}{j} =  \binom{b}{j-1} +  \binom{b}{j}$ for all $1 \leq j \leq b - 1$.
\end{proof}

\begin{corollary}\label{previouscorollary}
Let $f$ be a polynomial such that $\mbox{deg}_y(f) = 2$. Then $f\in L^{\delta}$ if and only if $f $ belongs to the subalgebra generated by $x, [y, x]$.
\end{corollary}
\begin{proof} We show the multihomogeneous case because the generalization follows directly from this result. Let $f$ be a multihomogeneous polynomial, $\mbox{deg}_y(f) = 2$ such that $f \in L^\delta$.
Since left normed  brackets generate $L'$, we assume that $f$ is of the form
$$f = \alpha_0 [y, x, y, \underbrace{x, \ldots, x}_{k}] + \alpha_1 [y, x, x, y, \underbrace{x, \ldots, x}_{k-1}] + \cdots + \alpha_k [y, x, \underbrace{x, \ldots, x}_{k}, y]$$
for some scalars $\alpha_i\in K$, $i=0,\ldots,k$. The fact that $\delta(f) = 0$ implies 
$$(\alpha_0 + \cdots + \alpha_k)[y, x, \underbrace{x, \ldots, x}_{k + 1}] = 0,$$
and therefore $\alpha_0 + \cdots + \alpha_k = 0$. The result follows from Lemma \ref{previouslemma}.
\end{proof}

In the next propositions, we show that the constants of degree less than or equal to $7$ are in the Lie subalgebra generated by $x, [y, x], [y, x, y, [y, x, x]]$.

\begin{proposition}If  $f \in L^\delta$  is a nonzero polynomial of degree less than or equal to $5$, then $f $ belongs to the subalgebra generated by $x, [y, x]$.
\end{proposition}
\begin{proof} Suppose without loss of generality that $f$ is a multihomogeneous polynomial.

If $\mbox{deg}_y(f) = 1$, $f $ belongs to the subalgebra generated by $x, [y, x]$ and consequently is a constant. If $\mbox{deg}_y(f) = 2$ the result follows from Corollary \ref{previouscorollary}.
Suppose $f$ homogeneous of degree $4$ such that $\mbox{deg}_y(f) = 3$. In this case $f = \alpha [y, x, y, y]$, for some $\alpha\in K$, and consequently $\delta(f) = 2 \alpha [y,x, x, y ] = 0$. This implies that $\alpha = 0$. 

Now suppose $f$ homogeneous of degree $5$. In this case it is sufficient to check two possibilities: $\mbox{deg}_y(f) = 3$ and $4$.
Assume that $\mbox{deg}_y(f) = 3$. Then
\[
f=\alpha [y, x, x, y, y] + \beta [y, x, y, y, x] \in L^\delta
\]
for some $\alpha,\beta\in K$. We have the following:
\begin{align}
0&=\delta(f)=(\alpha + 2\beta)[y, x, x, y, x] + \alpha[y, x, x, x, y]\nonumber\\
&= (\alpha + 2\beta)([y, x, x, x, y] + [y, x, x, [y, x]]) + \alpha[y, x, x, x, y].\nonumber
\end{align}
Since $[y, x, x, x, y]$ and $[y, x, x, [y, x]]$ belong to Hall's basis we have that $\alpha = \beta = 0$.
Now let $\mbox{deg}_y(f) =4$. Then $f=\alpha [y, x, y, y, y] \in L ^\delta$ for some $\alpha\in K$, that is, 
\[
\alpha [y, x, x, y, y] + \alpha [y, x, y, x, y] + \alpha [y, x, y, y, x] = 0.
\]
Since $[y, x, x, y, y] = [y, x, y, x, y]$ and $[y, x, y, y, x] = [y, x, y, x, y] + [y, x, y, [y, x]]$, the last equation implies that $3\alpha [y, x, x, y, y] + \alpha[y, x, y, [y,x]] = 0$. Therefore $\alpha = 0$ since $ [y, x, x, y, y]$, $[y, x, y, [y, x]]$ belong to Hall's basis of $L$ . 
\end{proof}

\begin{proposition}
If $f \in L^\delta$  is a polynomial of degree $6$, then $f $ belongs to the subalgebra generated by $x, [y, x]$ and $[y, x, y, [y, x, x]]$.
\end{proposition}
\begin{proof}As before, we may assume that $f$ is a multihomogeneous
polynomial. By Corollary \ref{previouscorollary}, there are three cases to check: $\mbox{deg}_y(f) = 3$, $4$ and $5$. Suppose $\alpha [y, x, y, y, y, y] \in L ^\delta$, i.e., $2\alpha [y, x, x, y, y, y] + \alpha [y, x, y, y, x, y] + \alpha [y, x, y, y, y, x] = 0$. Since 
$$[y, x, y, y, x, y] = [y, x, x, y, y, y] + [y, x, y, y, [y, x]]$$
and
$$[y, x, y, y, y, x] = [y, x, x, y, y, y] + 2[y, x, y, y, [y, x]],$$ 
it follows that $\alpha = 0$. Suppose
\begin{equation} \label{equacao_deg_y_4}
f = \alpha [y, x, x, y, y, y] + \beta [y, x, y, y, x, y] + \gamma [y, x, y, y, y, x] \in L^\delta, 
\end{equation}
that is, 
\begin{eqnarray*}
\delta(f) &=& \alpha [y, x, x, x, y, y] + (\alpha+ 2\beta) [y, x, x, y, x, y]\\
& &+ (\alpha + 2\gamma)[y, x, x, y, y, x]+ (\beta + \gamma)[y, x, y, y, x, x]=0.
\end{eqnarray*}
Since $v_1 = [y, x, x, x, y, y]$, $v_2 = [y, x, x, y, [y, x]]$, $v_3 = [y, x, y, [y, x, x]]$ belong to Hall's basis and by simple calculations using (\ref{identidade_importante}) 
\begin{align*}
[y, x, x, y, x, y] &= v_1 + v_2 - v_3,\\
[y, x, x, y, y, x] &= v_1 + 2v_2 - v_3,\\
[y, x, y, y, x, x] &= v_1 + 3v_2,
\end{align*}
we have the following system equations: $\alpha + \beta + \gamma = 0$ and $3 \alpha + 5 \beta + 7 \gamma = 0$  whose solution is $\beta = -2\alpha$ and $\gamma = \alpha$. Replacing this solution in (\ref{equacao_deg_y_4}) we get $f = - \alpha [y, x, y, [y, x, y]]  = 0$. Finally consider
\begin{equation} \label{equacao_deg_y_3}
f = \alpha [y, x, x, y, y, x] + \beta [y, x, y, y, x, x] + \gamma [y, x, x, x, y, y] + \epsilon [y, x, x, y, x, y]\in L^\delta, 
\end{equation}
that is, 
\[
(\alpha + \gamma) [y, x, x, x, y, x] + (\alpha + 2\beta + \epsilon)[y, x, x, y, x, x] + (\gamma + \epsilon)[y, x, x, x, x, y] = 0.
\]
Denote by $v_1 = [y, x, x, x, x, y]$ and $v_2 = [y, x, x, x, [y, x]]$. Since $[y, x, x, x, y, x] = v_1 + v_2$ and $[y, x, x, y, x, x] = v_1 + 2v_2$ follow that $\alpha + \beta + \gamma + \epsilon = 0$ and $3\alpha + 4\beta + \gamma + 2\epsilon = 0$. Consequently $\gamma = \alpha + 2\beta$ and $\epsilon = -2\alpha - 3\beta$. Replacing this solution in (\ref{equacao_deg_y_3}) we obtain 
\begin{eqnarray*}
f &=& \alpha([y, x, x, y, y, x] + [y, x, x, x, y, y] -2[y, x, x, y, x, y])\\
& &+ \beta ([y, x, y, y, x, x] + 2[y, x. x, x, y, y] - 3[y, x, x, y, x, y]) \in L^\delta.
\end{eqnarray*}
By simple calculations we conclude that 
$$[y, x, x, y, y, x] + [y, x, x, x, y, y] -2[y, x, x, y, x, y] = [y, x, y, [y, x, x]],$$ 
$$[y, x, y, y, x, x] + 2[y, x. x, x, y, y] - 3[y, x, x, y, x, y] = 3[y, x, y, [y, x, x]]$$
and this ends the proof.
\end{proof}

Next we prove that if $f \in L^\delta$  is a polynomial of degree $7$ then $f $ belongs to the subalgebra generated by  $x$, $[y, x]$, and $[y, x, x, y, [y, x, x]] - [y, x, x, x, [y, x, y]]$. We observe that $$[y, x, x, y, [y, x, x]] - [y, x, x, x, [y, x, y]] = [y, x, y, [y, x, x], x],$$ that is, the product between $[y, x, y, [y, x, x]]$ and $x$.

\begin{proposition}If $f \ne 0$ is a polynomial of degree $7$  then $f \in L^\delta$ if and only if $f $ belongs to the subalgebra generated by $x$, $[y, x]$, and $[y, x, x, x, [y, x, y]] - [y, x, x, y, [y, x, x]]$.
\end{proposition}

\begin{proof} As before, we may assume that $f$ is a multihomogeneous
polynomial. By Corollary \ref{previouscorollary}, we have four cases to check: $\mbox{deg}_y(f) = 3$, $4$, $5$ and $6$. 
\begin{enumerate}
\item $\mbox{deg}_y(f) = 3$. We write $f$ as a linear combination of the Hall's basis.  If we denote by 
\begin{align*}
v_1 &= [y, x, x, x, x, y, y],\\
v_2 &= [y, x, x, x, y, [y, x]],\\
v_3 &= [y, x, x, [y,x], [y, x]],\\
v_4 &= [y, x, x, x, [y, x, y]],\\
v_5 &= [y, x, x, y, [y, x, x]]
\end{align*}
and $f = \alpha_1 v_1 + \alpha_2 v_2  + \alpha_3 v_3 + \alpha_4 v_4 + \alpha_5 v_5$, $\delta(f) = 0$ implies that 
$$\alpha_1 \delta(v_1) + \alpha_2 \delta (v_2)  + \alpha_3 \delta (v_3) + \alpha_4 \delta (v_4)+ \alpha_5 \delta (v_5) = 0.$$ 
Since $\delta(v_1)  = 2[y, x, x, x, x, x, y ]  +  [y, x, x, x, x, [y, x]]$, $\delta(v_2) = [y, x, x, x, x, [y, x]]$, $\delta(v_3) = 0$, $\delta(v_4) = [y, x, x, x, [y, x, x]]$ and $\delta(v_5) = [y, x, x, x, [y, x, x]]$ we have that $\alpha_1 = \alpha_2 = 0$ and $\alpha_4 + \alpha_5 = 0$. Therefore $f \in  L ^\delta$ if and only if $f$ is a linear combination of $[y, x, x, [y, x],[y, x]]$ and $[y, x, x, x, [y, x, y]] - [y, x, x, y, [y, x, x]]$.
 
\item $\mbox{deg}_y(f) = 4$.  Denote by 
\begin{align*}
v_1 &= [y, x, x, x, y, y, y],\\
v_2 &= [y, x, x, y, y, [y, x]],\\
v_3 &= [y, x, y, [y,x], [y, x]],\\
v_4 &= [y, x, x, y, [y, x, y]],\\
v_5 &= [y, x, y, y, [y, x, x]],
\end{align*}
we have that $f = \alpha_1 v_1 + \alpha_2 v_2  + \alpha_3 v_3 + \alpha_4 v_4 + \alpha_5 v_5$. Therefore $\delta(f) = 0$ implies that 
$$\alpha_1 \delta(v_1) + \alpha_2 \delta (v_2)  + \alpha_3 \delta (v_3) + \alpha_4 \delta (v_4)+ \alpha_5 \delta (v_5) = 0.$$ 
Since  
\begin{align*}
\delta(v_1)  &= 3 [y, x, x, x, x, y, y] + 3[y, x, x, x, y, [y, x]] + 2[y, x, x, x, [y, x, y]],\\
\delta(v_2) &= 2[y, x, x, x, y, [y, x]] + [y, x, x, [y, x], [y, x]],\\ \delta(v_3) &=  [y, x, x, [y, x], [y, x]],\\
\delta(v_4) &= [y, x, x, x, [y, x, y]] + [y, x, x, y, [y, x, x]],\\
\delta(v_5) &= 2[y, x, x, y, [y, x, x]],
\end{align*}
we have that $\alpha_i =  0$ for all $i \in \{1, 2, 3, 4, 5\}$.

\item $\mbox{deg}_y(f) = 5$.   We write 
\begin{align*}
v_1 &= [y, x, x, y, y, y, y],\\
v_2 &= [y, x, y, y, y, [y, x]],\\
v_3 &= [y, x, y, y, [y, x, y]],
\end{align*}
and $f = \alpha_1 v_1 + \alpha_2 v_2  + \alpha_3 v_3$. Therefore $\delta(f) = 0$ implies that 
$$\alpha_1 \delta(v_1) + \alpha_2 \delta (v_2)  + \alpha_3 \delta (v_3) = 0.$$ 
Since 
\begin{align*}
\delta(v_1) 
&= 4[y, x, x, x, y, y, y] 
+ 6[y, x, x, y, y, [y, x, y]] \\
&\quad + 8[y, x, x, y, [y, x, y]] 
- 3[y, x, y, y, [y, x, x]], \\
\delta(v_2) 
&= 3[y, x, x, y, y, [y, x]] 
+ [y, x, y, [y, x], [y, x]], \\
\delta(v_3) 
&= 2[y, x, x, y, [y, x, y]] 
+ [y, x, y, y, [y, x, x]],
\end{align*}
 we have that  $\alpha_i =  0$ for all $i$.

\item $\mbox{deg}_y(f) = 6$. Let $f = \alpha [y, x, y, y, y, y, y]  \in L^\delta$, that is, 
\[
0 = \alpha(2 [y, x, x, y, y, y, y]+ [y, x, y, y, x, y, y] + [y, x, y, y, y, x, y] + [y, x, y, y, y, y, x]).
\]
By simple calculations 
\[
2 [y, x, x, y, y, y, y] + [y, x, y, y, x, y, y] + [y, x, y, y, y, x, y] + [y, x, y, y, y, y, x]
\]
is a non zero linear combination of the Hall's basis $[y, x, x, y, y, y, y]$, $[y, x, y, y, y, [y, x]]$ and $[y, x, y, y, [y, x, y]]$. Thus $\alpha = 0$.
\end{enumerate}
\vspace{-.77cm}
\end{proof}

\begin{corollary}  If $f \neq 0$ is a multihomogeneous polynomial such that $\mbox{deg}_y(f) = 3$ and $\mbox{deg}(f) = n$, then $f \in L^{\delta}$ if and only if $f $ belongs to the subalgebra generated by $x$, $[y, x]$ and 
\[
[y, \underbrace{x, \ldots, x}_k, [y, \underbrace{x, \ldots, x}_{n-3-k}, y]] - [y,\underbrace{x, \ldots, x}_{k-1}, y, [y, \underbrace{x, \ldots, x}_{n-2-k}]]
\]
where $k$ is an integer and $\frac{n-2}{2}\leq k \leq n-4$.
\end{corollary}
\begin{proof} Suppose $f \neq 0$ a multihomogeneous polynomial such that $\mbox{deg}(f) = n$. It is sufficient to show when $\mbox{deg}_y(f) > 7$. We write $f$ as a linear combination of the Hall's basis, that is, as a linear combination of polynomials 
\begin{align*}
&\quad \ [y, \underbrace{x, \ldots, x}_{n-3}, y, y],\\
f_{k} &= [y, \underbrace{x, \ldots, x}_{k}, y, [y, \underbrace{x, \ldots, x}_{n-3-k}]],\\
f_{l. m} &= [y, \underbrace{x, \ldots, x}_{n-3-l-m}, [y, \underbrace{x, \ldots, x}_{l}], [y, \underbrace{x, \ldots, x}_{m}]],\\ 
f_p &= [y, \underbrace{x, \ldots, x}_{p}, [y,  \underbrace{x, \ldots, x}_{n-3 - p}, y]]
\end{align*}
for all $k$, $l$, $m$, $p$ integers $\geq 1$ such that $\frac{n-4}{2}\leq k \leq n-4$, $\frac{n-2}{2}\leq p \leq n-4$, $l \leq m < \frac{n-3}{2}$ and $2l + m < n-3$. The proof is analogous to the previous proposition when $\mbox{deg}_y(f) = 3$.
\end{proof}

\subsection{Proof using invariant theory}
We start recalling the next useful result by Drensky and Gupta (see Theorem 4.6 of \cite{DG}).

\begin{theorem}
For any $GL_2(K)$-invariant ideal $I$ of $K\langle x,y \rangle$ the algebra of constants
$(K\langle x, y\rangle/I )^\delta$ is spanned by the highest weight vectors of the $GL_2(K)$-irreducible components
of $K\langle x, y\rangle/I$.
\end{theorem}

Keeping in mind the previous result, let $L$ be the free Lie algebra of rank two generated by $x,y$. If $L$ is a direct sum of irreducible $GL_2(K)$-modules $W(\lambda_1,\lambda_2)$, then in each 
summand $W(\lambda_1,\lambda_2)$ there is a homogeneous polynomial of degree $\lambda_1$ in $x$ and $\lambda_2$ in $y$ and these polynomials form a basis of $L^\delta$. If a constant is of degree $\lambda_1$ in $x$ and $\lambda_2$ in $y$, then it belongs to $W(\lambda_1,\lambda_2)$. In particular, we get the homogeneous constants $f$ in $x$ and $y$ do satisfy the 
inequality $\deg_x(f)\geq \deg_y(f)$. Now, if we know the multiplicity $m{(\lambda_1,\lambda_2)}$ of the $GL_2(K)-$module $W(\lambda_1,\lambda_2)$ in $L$, then we only need to find $m{(\lambda_1,\lambda_2)}$ linearly independent homogeneous constants of degree $\lambda_1$ in $x$ and $\lambda_2$ in $y$. On this purpose, it is crucial the knowledge of the decomposition of the homogeneous 
components of the free Lie algebra of any finite rank that is well known for degrees up to 10 (see \cite{Thrall}). In particular, if $L_n$ denotes the homogeneous component of degree $n$ of $L$, we have:
\[L_1=W(1),\ L_2=W(1,1),\ L_3=W(2,1),\ L_4=W(3,1),\ L_5=W(4,1)+W(3,2),\]\[L_6=W(5,1)+W(4,2)+W(3,3),\ L_7=W(6,1)+2W(5,2)+2W(4,3).\]
Now remark $L'/L''=\oplus_{n\geq 2}W(n-1,1)$ and therein the constants are monomials of type $[y,x,\ldots,x]$. For the other modules, the following polynomials form bases of the homogeneous 
components $L^{\delta}(\lambda_1,\lambda_2)$ with obvious meaning of the notation:
\[
(n - 2,2): [[y,x,…,x],[y,x,…,x]],
\]
\[
(3,3): [[y,x,y],[y,x,x]],
\] 
\[
(4,3): [[y,x,x],[y,x],[y,x]], [[y,x,x,y],[y,x,x]] - [[y,x,x,x],[y,x,y]].
\]
Observe the modules $W(n-2,2)$ belong to $L''$, then the constants can only be polynomials of type
$[[y,x,\ldots,x],[y,x,\ldots,x]]$ which belongs to the subalgebra $L_{\leq 2}$ of $L$ generated by $x,[y,x]$. Notice that the polynomial $[[y,x,y],[y,x,x]]$ does not belong to $L_{\leq 2}$ whereas it is easy to see that 
\[[[y,x,x,y],[y,x,x]] - [[y,x,x,x],[y,x,y]] = [[[y,x,y],[y,x,x]],x].\]
 
 We have proved the following.

\begin{theorem}
The constants of $L$ of degree less than or equal to 7 belong to the Lie subalgebra of $L$ generated by \[x,[y,x],[[y,x,y],[y,x,x]].\] 
\end{theorem}

Indeed, if $A:=[x,y,y]$, $B:=[x,y,x]$, then $\delta(A)=B$ and 
\[
M=[A, B]=-2 U_{A, A}^{(1,0)} \in U^{(1)}.
\]
Because of the previous experimental results, we can state the next conjecture.

\begin{conjecture}
Let $M$ be a Hall monomial so that it is a constant under the derivation $\delta$, then $M\in U^{(1)}$. 
\end{conjecture}

We can also state a stronger conjecture supported by the results of this section.

\begin{conjecture}
The algebra $L^\delta$ of constants of $\delta$ in the free Lie algebra $L$ of rank $2$ is the algebra $U^{(1)}$.
\end{conjecture}






\EditInfo{January 8, 2025}{May 14, 2025}{David Towers and Ivan Kaygorodov}

\end{document}